\documentclass[11pt, oneside]{article}   	
\usepackage{geometry}                		
\usepackage{graphicx}				
\usepackage{amssymb}
\usepackage{hyperref}

\usepackage{geometry}                		
\usepackage{graphicx}				
\usepackage{xypic}
\usepackage{tikz}
\usepackage{amssymb,amsthm,amsmath}
\usepackage{hyperref}
\usepackage{float}
\usepackage{enumitem}

\newtheorem{theorem}{Theorem}
\numberwithin{theorem}{section}

\newtheorem{lemma}[theorem]{Lemma}
\newtheorem{corollary}[theorem]{Corollary}
\newtheorem{definition}[theorem]{Definition}

\newtheorem{question}[theorem]{Question}

\newtheorem*{proposition*}{Proposition}
\newtheorem*{theorem*}{Theorem}
\newtheorem*{corollary*}{Corollary}

\theoremstyle{definition}
\newtheorem{example}[theorem]{Example}

\theoremstyle{remark}
\newtheorem{remark}[theorem]{Remark}

\usetikzlibrary{calc}
\usetikzlibrary{positioning}
\tikzset{me/.style={to path={
\pgfextra{%
 \pgfmathsetmacro{\startf}{-(#1-1)/2}  
 \pgfmathsetmacro{\endf}{-\startf} 
 \pgfmathsetmacro{\stepf}{\startf+1}}
 \ifnum 1=#1 -- (\tikztotarget)  \else
     let \p{mid}=($(\tikztostart)!0.5!(\tikztotarget)$) 
         in
\foreach \i in {\startf,\stepf,...,\endf}
    {%
     (\tikztostart) .. controls ($ (\p{mid})!\i*6pt!90:(\tikztotarget) $) .. (\tikztotarget)
      }
      \fi   
     \tikztonodes
}}}

\tikzset{main node/.style={circle,fill=blue!20,draw,minimum size=1.3cm,inner sep=0pt},
            }

\title{A note on algebraic rank, matroids, and metrized complexes}
\author{Yoav Len}

\begin{document}

\maketitle
\begin{abstract}
We show that the algebraic rank of divisors on certain graphs is related to the realizability problem of matroids. 
As a consequence, we  produce a series of examples in which the algebraic rank depends on the ground field. We use the theory of metrized complexes to show that equality between the algebraic and combinatorial rank is not a sufficient condition for smoothability of divisors, thus giving a negative answer to a question posed by Caporaso, Melo, and the author. 
\end{abstract}
\section{Introduction}
In \cite{Caporaso}, Caporaso defines the \emph{algebraic rank}, an invariant on graphs which reflects the ranks of line bundles with prescribed degrees on all the nodal curves dual to the given graph. This invariant was studied in \cite{CLM}, and was shown to be bounded above by the  combinatorial rank of divisors on graphs (in the sense of \cite{BN,AC}). 

In this note, we further study the algebraic rank. Using a construction of Cartwright from \cite{Cartwright}, we assign a finite graph $G_M$ and a divisor $D_M$ to any simple matroid of rank $3$, where $G_M$ is the Levi graph of $M$  (see Definition~\ref{matroidGraph}).
Our main result is the following:

\begin{theorem*} [\ref{rankMatroid}]
Let $k$ be an algebraically closed field. Then $D_M$ has algebraic rank $2$ over $k$ if and only if $M$ is realizable over $k$.
\end{theorem*} 
\noindent Since there are matroids that are realizable over a certain field but not another, we conclude:

\begin{corollary*}[\ref{fieldDependence}]
There exists a divisor whose algebraic rank depends on the field. 
\end{corollary*}

On a different vein, as the algebraic rank is bounded from above by the combinatorial rank, it is natural to ask how it compares with other refinements of the combinatorial rank. In Section~\ref{sec:CLMMetrized}, we examine the relation between the rank of divisors on metrized complexes, and the algebraic rank of divisors obtained by forgetting the metrized complex structure. We show, using two examples, that in general, there is no inequality between one and the other. As an application, we obtain a negative answer to the following question, that originally appeared in \cite[Question 2.11]{CLM}:

\begin{question}\label{smoothable}
Assume that the algebraic and combinatorial rank of a divisor $D$ coincide. Is   $D$ smoothable?
\end{question}

\subsection{Algebraic rank}
In what follows, we assume familiarity with  divisor theory on finite vertex-weighted graphs, see \cite{BN,AC} for an exposition. Throughout this note, we shall refer to the standard rank of divisors on graphs (defined in \textit{loc. cit.}) as the \emph{combinatorial rank}. 

The \emph{algebraic rank} is an invariant of divisor classes, first introduced in \cite{Caporaso}. Given a divisor $D$ on a finite graph $G$, the invariant should reflect the ranks of line bundles on nodal curves dual to $G$, whose degrees on each component are prescribed by $D$. That is, the degree of the restriction to each component equals the coefficient of $D$ at the corresponding vertex.
However, this involves making several choices:
\begin{enumerate}
\item The graph $G$ admits many different dual curves.
\item The invariant is intended to be well defined on divisor classes, but there is no canonical choice of representative in the divisor class of $D$.
\item There are many different line bundles with degrees prescribed by the divisors in the class. 
\end{enumerate}
The algebraic rank is made free of all choices, by defining it as  
$$
r^{\text{alg}}_k(D) = \max_{X\in M_k^\text{alg}(G)}\Biggl[\min_{\overset{}{E\simeq D}}\biggl(\max_{L\in\text{Pic}^{E}(X)}(h^0(L)-1)\biggr)\Biggr],
$$
\noindent where  $M_k^\text{alg}(G)$ is the set of isomorphism classes of curves defined over $k$ having $G$ as their dual graph, and $\text{Pic}^{E}(X)$ is the set of line bundles on $X$ whose multidegree is prescribed by $E$.

\noindent As shown in  \cite{CLM}, the algebraic rank satisfies the following familiar properties:
\begin{theorem}\label{thm:CLM}
The algebraic rank of a divisor on a finite graph is bounded above by the combinatorial rank, satisfies a Riemann-Roch equality, and a specialization lemma. 
\end{theorem}

Let $R$ be a discrete valuation ring with fraction field $K$ and algebraically closed residue field $k$. 
Let $X$ be a smooth projective curve over $K$, and $\mathcal{X}$ a regular semistable model for $X$. That is,  $\mathcal{X}$ is a proper, flat, and regular scheme over $R$, whose generic fiber is isomorphic to $X$, and whose special fiber is a reduced nodal curve $X_0$ over $k$. Combining the theorem above with Baker's Specialization Lemma (\cite{BN,AC}), we get:
\begin{corollary}\label{specializationRefinement}
Let $D$ be a divisor on $X$ that specializes to a divisor $D_G$ on $G$. Then

$$
r(D) \leq r^{\text{alg}}_k(D_G)\leq r_G(D_G),
$$
where $r_G$ stands for the combinatorial rank. 
\end{corollary}

\begin{definition}\label{def:smoothable}
A divisor $D$ as above is said to be a \emph{lifting} of $D_G$. If $D_G$ admits a lifting to a divisor $D$ with $r(D)=r_G(D_G)$, then we say that $D_G$ is \emph{smoothable}. 
\end{definition}


\section{Dependence of the algebraic rank on the ground field}\label{sec:CLMField}
In this section we show that a divisor can have different algebraic ranks over different (algebraically closed) fields. For that purpose, we appeal to matroid theory. 
Rather than defining matroids in general, we only define here simple matroids of rank $3$. See \cite{Ox} for a comprehensive treatment of matroids of any rank.

\begin{definition}\label{simpleMatroid}
A rank $3$ \emph{simple matroid} $M$ consists of a finite set $\mathcal{E}$ and a collection of subsets $\mathcal{F}$ satisfying:
\begin{enumerate}
\item Every pair of distinct elements $E_1$ and $E_2$ in $\mathcal{E}$ are contained in exactly one element of $\mathcal{F}$.
\item $\mathcal{F}$ contains at least two elements. 
\end{enumerate}
We refer to every $E$ in $\mathcal{E}$ as an \emph{element} of the matroid, and to $F$ in $\mathcal{F}$ as a \emph{flat}. The unique flat containing a pair $E_i$ and $E_j$ of elements will be denoted $F_{ij}$.
\end{definition}

In \cite{Cartwright}, Cartwright assigns a graph and a divisor to every simple matroid of rank $3$  as follows:
\begin{definition}\label{matroidGraph}
Let $M$ be a simple matroid of rank $3$. Let $G_M$ be the incidence graph of $M$. Namely, $G_M$ has a vertex $v_E$ for each element $E$, and a vertex $v_F$ for each flat $F$ of $M$. There is an edge between $v_E$ and $v_F$ whenever $E$ is contained in $F$.
We define $D_M$ to be the divisor with a single chip on every vertex $v_E$ corresponding to an element  of $M$.
\end{definition}
\noindent As shown in \cite[Proposition 2.2]{Cartwright}, the combinatorial rank of $D_M$ is 2.
A similar construction was used independently by Castravet and Tevelev  to study  planar realization of hypertrees  \cite{Hypertrees}.  

\begin{definition}
A rank $3$ simple matroid is \emph{realizable} over $k$ if there is a bijection between its elements and a collection $\{\ell_1,\ldots,\ell_m\}$ of lines in $\mathbb{P}^2_k$, such that $\ell_{i_1},\ldots,\ell_{i_k}$ have non-empty intersection  if and only if the corresponding elements of $M$ are contained in a flat. 
\end{definition}

The rest of this section is devoted to proving the following result:
\begin{theorem} \label{rankMatroid}
Let $M$ be a simple matroid of rank $3$. Then $D_M$ has algebraic rank $2$ over $k$ if and only if $M$ is realizable over $k$.
\end{theorem} 

\begin{corollary}\label{fieldDependence}
There exists a divisor whose algebraic rank depends on the field. 
\end{corollary}
\begin{proof}
Let $M$ be the Fano matroid, consisting of seven elements, such that every three of them are contained in a unique flat (Figure~\ref{Fano}). Then it is known that $M$ is realizable over  $k$ if and only if the characteristic of $k$ is $2$ \cite[Example 4.5]{KatzMatroids}. Therefore, Theorem~\ref{rankMatroid} implies that 
$r^\text{alg}_{\mathbb{C}}(D_M) $ is at most $1$,  but   $r^\text{alg}_{\bar{\mathbb{F}}_2}(D_M) = 2$.

\begin{figure}[h]
\centering

\begin{tikzpicture}[scale=.5]

 \coordinate (A) at (0,0); 
    \coordinate (B) at (6,0);
    \coordinate (C) at ($ (A)!1!60:(B) $); 

    \coordinate (AB) at ($ (A)!.5!(B) $); 
    \coordinate (AC) at ($ (A)!.5!(C) $); 
    \coordinate (BC) at ($ (B)!.5!(C) $); 
    \coordinate (ABC) at ($ (A)!2.0/3.0!(BC) $); 

 \draw [thick] (A) -- (B);
 \draw [thick] (B) -- (C);
 \draw [thick] (C) -- (A);
 \draw [thick] (A) -- (BC);
 \draw [thick] (B) -- (AC);
 \draw [thick] (C) -- (AB);

\draw [thick] (ABC) circle (1.75cm); 
    \fill (A) circle (6pt);
    \fill (B) circle (6pt);
    \fill (C) circle (6pt);
    \fill (AB) circle (6pt);
    \fill (AC) circle (6pt);
    \fill (BC) circle (6pt);
    \fill (ABC) circle (6pt);
    
\end{tikzpicture}   

\caption[The Fano matroid]{The Fano Plane. The lines correspond to the elements of the Fano matroid, and the marked points correspond to the flats.}
\label{Fano}
\end{figure}
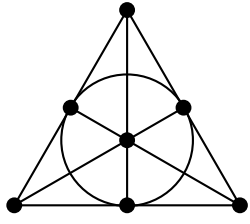
\end{proof}

\noindent The graph $G_M$ from the corollary is known as the \emph{Heawood graph}. A similar graph but with generic edge lengths was used by Jensen to show that the locus  of Brill--Noether general graphs is not dense in the moduli space of tropical curves \cite{Jensen}. 

Let $(L,H)$ be a linear series on a nodal curve $X$. Then $H$ has a \emph{basepoint} at a point $p$ if all the sections of $H$ vanish at $p$. 
\begin{lemma}\label{nonDegMap}
Let $H$ be a linear series of rank $r\geq 0$ on a nodal curve $X$, and let $Y$ be the union of the components where not all the global sections of $H$ vanish identically. Then $H$ induces a non-degenerate rational map from the smooth locus of $Y$ to $\mathbb{P}^r$.
\end{lemma}
\begin{proof}
Let  $\pi:\tilde{Y}\to Y$ be the normalization of $Y$. Then the pullback to $\tilde{Y}$ of the sections of $H$ induce a rational map $\tilde{\phi}:\tilde{Y}\dashrightarrow\mathbb{P}^r$. For every smooth point $p$ of $Y$ such that $\tilde\phi$ is defined on $\pi^{-1}{p}$, let  $\phi(p) = \tilde\phi(\pi^{-1}{p})$.
To see that the map is non-degenerate, suppose by contradiction that the image of $\phi$ is contained in a hyperplane whose defining equation is $a_0\cdot x_0 + \cdots + a_r\cdot x_r = 0$. Then $a_0\cdot\sigma_0 +\cdots+a_r\cdot\sigma_r=0$, contradicting the fact that  $\{\sigma_i\}$ is a basis.  
\end{proof}

\begin{remark}
If $H$ has no basepoint at a node of $X$, then $\phi$ may be extended to that point.
\end{remark}

\begin{example}
Let $X$ be a union of three rational curves $A,B$ and $C$, such that each pair meets at a single point. Let $L$ be a line bundle of degree $1$ on each component, and suppose that it has a single basepoint at the intersection of $A$ and $B$. Then the map $\phi$ which corresponds to the complete linear system of $L$ is defined everywhere away from the basepoint, and maps $C$ isomorphically onto a line. Since there is a basepoint on $A$ and $B$, their images $\phi(A)$ and $\phi(B)$ are single points.  Those points must be distinct, as $\phi$ is an isomorphism on $C$.
\end{example}

Elements $E_1,E_2,E_3$ of $M$ are a \emph{basis} if $F_{12}, F_{23}, F_{13}$  are distinct. A basic fact from matroid theory is: 
\begin{lemma}\label{matroidBasis}
Every pair of distinct elements can be completed to a basis. 
If $E_1,E_2,E_3$ is a basis, then every element $E_0$ forms a basis with at least two of them. 
\end{lemma}
\begin{proof}
As the matroid is simple, every distinct pair of elements $E_1,E_2$ is independent \cite[Definition 3.8]{KatzMatroids}. The size of any maximally independent set equals the rank which, by assumption, is three, so $E_1,E_2$ are contained in a basis. 
%
The second part follows immediately by extending $E_0$ to a basis, and using the matroid exchange property \cite[Definition 3.6]{KatzMatroids}. 
\end{proof}

We now turn to prove Theorem~\ref{rankMatroid}.
\begin{proof}
Let $G_M$ and $D_M$ be as in Definition~\ref{matroidGraph}.
If $M$ is realizable over $k$, then by \cite[Theorem 1.5]{Cartwright}, $D_M$ is smoothable over $k[[t]]$ in the sense of  Definition~\ref{def:smoothable}. That is, there is a curve over $k[[t]]$ whose generic fiber is a smooth $k((t))$-curve $X$, whose special fiber has dual graph $G_M$, and there is a divisor $D$ of rank $2$ on $X$ that specializes to $D_M$. By Corollary~\ref{specializationRefinement}, 
$$
2=r(D)\leq r^\text{alg}_k(D_M)\leq r_{G_M}(D_M)=2,
$$
so the algebraic rank of $D_M$ over $k$ is $2$.

Conversely, suppose that $r^{\text{alg}}_k(D_M)=2$. From the definition of the algebraic rank, there exists a curve $X_M$ dual to $G_M$ and a line bundle $L$ of rank at least $2$ whose degrees are prescribed by $D_M$. Let $(L,H)$ be a linear series such that $\dim(H)=3$, and $\phi_M$ the non-degenerate rational map to $\mathbb{P}^2$  defined in Lemma \ref{nonDegMap}.
 For an element $E$ of $M$, we denote by $v_{E}$ the corresponding vertex of $G_M$, and  $C^{E}$  the corresponding component of $X_M$. Similarly, for every flat $F$, we denote by $v_F$ and $C^F$  the corresponding vertex and component. In what follows, we will show that $\phi_M$ has no basepoints, maps each $C^E$ isomorphically onto a line, and those lines intersect as prescribed by $M$.

If $H$ has a basepoint on a component $C^F$ for some flat $F$ then,  as $D_M$ has degree zero on $v_f$, all the sections of $H$ vanish identically on it, and $\phi_M$ is not defined on that component. If $H$ has no basepoints on $C^F$, then $\phi_M$ maps it to a point. Similarly, if $H$ has more than one basepoint on a component $C^E$ then the sections of $H$ vanish identically on it, and if it has a single basepoint then $C^E$ is mapped by $\phi_M$ to a point.
Therefore, since $\phi_M$ is non-degenerate, there must be an element  $E_1$ such that $H$ has no basepoints on $C^{E_1}$, and $\phi_M$ maps it isomorphically onto a line $\ell_1$.  We claim, moreover, that  $\phi_M$ is defined on every component. Otherwise, let $E_2$ be such that  all the global sections of $H$ vanish on $C^{E_2}$. Then they also vanish on $C^{F_{12}}$, so $H$ has a basepoint on $C^{E_1}$, a contradiction. We conclude that $\phi_M$ is defined on every component, and either maps it to a point on $\ell_1$ or a line that intersects $\ell_1$. 

Next, we wish to show that, in fact, $\phi_M$ has not basepoints at all. By the non-degeneracy of the map, there is an element $E_2$ such that $\phi_M(C^{E_2})$ is another line $\ell_2$ meeting $\ell_1$ at a point $p$. Now, suppose for the sake of contradiction that $E_3$ is such that $H$ has a basepoint on $C^{E_3}$. Then $\phi_M(C^{E_3})$ is a point $q$ which is both on $\ell_1$ and $\ell_2$. The point $q$ must be different from $p$ since $\phi_M$ is an isomorphism on $C^{E_1}$ and $C^{E_2}$. But then, $p$ and $q$ are two different points lying on $\ell_1$ and $\ell_2$, contradicting the fact that they are different lines.  Therefore, $\phi_M$ maps every $C^{E}$ isomorphically onto a line.
 
It remains to show that each of these components maps to a distinct line. Suppose for contradiction that there are two components $C^{E'}, C^{E''}$ that map to the same line, and choose  $E'''$ such that $E',E'',E'''$ is a basis for $M$. As $\phi_M$ is one-to-one on each component, it follows that $C^{E'''}$ is mapped onto the same line as well. Now, let $E$ be any element of $M$. By Lemma \ref{matroidBasis}, $E$ forms a basis with two of the elements in $\{E',E'',E'''\}$, and therefore, $C^E$ is mapped to the same line. But this contradicts the non-degeneracy of $\phi_M$.

We have seen that $\phi_M$ maps each $C^F$ to a point, and each $C^E$ isomorphically onto a line. Therefore, every pair $E_1$ and $E_2$ of elements of $M$ correspond to distinct lines  $\phi_M(C^{E_1})$ and $\phi_M(C^{E_2})$  meeting at $\phi_M(C^{F_{12}})$. Since $\phi_M$ is one-to-one on each of them, their images cannot meet anywhere else. Therefore, the image of $\phi_M$ is a realization of $M$, and the proof is complete.
\end{proof}

\begin{remark}
Our theorem implies that when a simple matroid $M$ of rank $3$ is not realizable over $k$, the corresponding divisor $D_M$ cannot be smoothed over the fraction field of any discrete valuation ring whose residue field is $k$. This is a slight strengthening of  \cite[Theorem 1.5]{Cartwright}, which implies that $D_M$ is not smoothable over $k[[t]]$.
\end{remark}


\section{Relation with  metrized complexes}\label{sec:CLMMetrized}
In \cite{AB}, Amini and Baker develop the theory of metrized complexes, a generalization of tropical curves. Roughly speaking, a metrized complex is obtained from a tropical curve by placing smooth curves at the vertices of the graph, and defining linear equivalence in a way that combines chip firing on the graph with linear equivalence on the curves. We refer the reader to \cite{AB}  for an exposition of metrized complexes and their divisor theory. 

The graphs in this section are finite vertex-weighted graphs, and the edges of all the metrized complexes are assumed to have length $1$. For a divisor $\mathcal{D}$, we denote  its restriction to a component $C$ by $\mathcal{D}|_C$. For the rest of the section, fix an algebraically closed field $k$.

\begin{definition}
Let $\mathcal{D}$ be a divisor on a metrized complex $C$ with underlying graph $G$. If $D$ is the divisor obtained from $\mathcal{D}$ by forgetting the metrized complex structure, we say that $\mathcal{D}$ is an \emph{extension} of $D$. 
\end{definition}

Suppose that $\mathcal{D}$ is an extension of $D$. Then both $r^\text{alg}_k(D)$ and $r(\mathcal{D})$ are refinements of the combinatorial rank, obtained by associating algebraic curves to the vertices, and considering line bundles on them.
Therefore, we ask whether they are related. Our main result in this section is that, in general, there is no inequality between the ranks in either direction.

In the following example, the algebraic rank of a divisor is strictly greater than the rank of {\bf every} extension to {\bf any} metrized complex.
\begin{example} 
Let $G$ be a graph with three vertices $u,v,w$, such that there is a single edge between $u$ and $v$, and a single edge between $v$ and $w$. The weight of each vertex is set to be $1$. Let $D$ be the divisor $u + w$  (see Figure \ref{metrizedEx1}). 

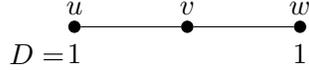
\begin{figure}[H]
\centering
\begin{tikzpicture}[scale=.5]
\begin{scope}[shift={($(2.5,-2.5)$)}]
\node at (0,0) {$u$};
\node [below] at (-1,-0.7) {$ D = $};
\node [below] at (0,-0.7) {$1$};
\node at (3,0) {$v$};
\node [below] at (3,-0.7) {$$};
\node at (6,0) {$w$};
\node [below] at (6,-0.7) {$1$};

\node[circle,fill=black,inner sep=1.5pt,draw] (a) at (0,-0.5) {};
\node[circle,fill=black,inner sep=1.5pt,draw] (b) at (3,-0.5) {};
\node[circle,fill=black,inner sep=1.5pt,draw] (c) at (6,-0.5) {};
\draw (a) edge[me=1] (b); 
\draw (c) edge[me=1] (b); 
\end{scope} 

\end{tikzpicture}   
\caption[The graph pertaining to example \ref{metrizedEx1}.] {The graph $G$.}
\label{metrizedEx1}
\end{figure}
\noindent The combinatorial rank of $D$ is $1$. Since $G$ is a hyperelliptic graph, it follows from \cite[Theorem 1.1]{KY2} that $r^\text{alg}_k(D) = 1$ as well. Now, let $\mathcal{D}$ be any extension of $D$ to a metrized complex $\mathcal{C}$. That is, $\mathcal{C}$ consists of three components $C_u$, $C_v$, and $C_w$ of genus $1$. There is single edge connecting $C_u$ and $C_v$, and a single edge connecting $C_v$ and $C_w$. $\mathcal{D}$ contains a single chip on $C_u$ and a single chip on $C_w$. Choose a point $x$ on $C_u$ where $\mathcal{D}$ does not have a chip. Then $\mathcal{D} - x$ is not linearly equivalent to an effective divisor. Therefore, $r_\mathcal{C}(\mathcal{D}) = 0$.
 
\end{example}

We conclude that the answer to Question \ref{smoothable} is negative:
\begin{corollary}
There exists a non-smoothable divisor $D$ whose algebraic and combinatorial ranks are equal. 
\end{corollary}
\begin{proof}
Let $G$ and $D$ be as in the example above, so that $r_k^{\text{alg}}(D) = r_G(D) = 1$. Assume for contradiction  that $D$  is smoothable. That is, there exists a semistable model with smooth generic fiber $X$ and special fiber $X_0$ whose weighted dual graph is $G$, and a divisor $D_X$ on $X$ of rank $1$ whose specialization is $D$. Let $\mathcal{C}$ be the regularization of $X_0$ (see Section 1 of \cite{AB}), and $\mathcal{D}$ the specialization of $D$ to $\mathcal{C}$. By the Specialization Lemma for metrized complexes (\cite[Theorem 4.6]{AB}), $r_\mathcal{C}(\mathcal{D})\geq  r(D)= 1$. But as  shown above, $\mathcal{D}$ has rank zero since it is an extension of $D$.

\end{proof}

Next, we deal with the converse inequality by finding a divisor $D$ whose algebraic rank is strictly smaller than the rank of \textbf{any} of its extensions to metrized complexes. We will, in fact, prove a stronger result: every extension of $D$ can be chosen to be a limit $\mathfrak{g}^2_d$.  The idea originated from a question posed to the author by Matt Baker at the Joint Mathematical Meetings in Baltimore. 
We remind the reader of the definition of a limit $\mathfrak{g}^r_d$ on a metrized complex.

\begin{definition}[\cite{AB}]
Let $\mathcal{C}$ be a metrized complex. For each vertex $v$, choose a $k$-linear subspace $F_v$ of $k(C_v)$, and let $\mathcal{F}$ be the collection of these spaces. The $\mathcal{F}$-rank of a divisor $\mathcal{D}$ is the largest integer $r$ such that for any effective divisor $\mathcal{E}$ of degree $r$, there exists a rational function $\mathfrak{f}$ 
on $\mathcal{C}$ whose $C_v$-part belongs to $F_v$ for all $v$, such that $\mathcal{D}-\mathcal{E} + \text{div}(\mathfrak{f})$ is effective. 

If $\mathcal{F}$ can be chosen so that each $F_v$ is a subspace of dimension $r+1$ (where $r$ is the $\mathcal{F}$-rank of $\mathcal{D}$) of the global sections of some line bundle $L_v$ of degree $d$, then $\mathcal{D}$ is said to be a limit $\mathfrak{g}^r_d$.
\end{definition}

\noindent Clearly, if a divisor is a limit $\mathfrak{g}^r_d$, then its rank is at least $r$.

\begin{example}
Let $G$ and $D$ be the graph and divisor in Figure \ref{metrizedGraph} with all weights equal to zero.

\begin{figure}[h!]\label{metrizedGraph}
\centering
\begin{tikzpicture}[scale=.5]

\begin{scope}[shift={($(2.5,-2.5)$)}]
\node at (0,0) {$v_1$};
\node [below] at (-1,-0.7) {$ D = $};
\node [below] at (0,-0.7) {$1$};
\node at (3,0) {$v_3$};
\node [below] at (3,-0.7) {$3$};
\node at (6,0) {$v_2$};
\node [below] at (6,-0.7) {$2$};

\node[circle,fill=black,inner sep=1.5pt,draw] (a) at (0,-0.5) {};
\node[circle,fill=black,inner sep=1.5pt,draw] (b) at (3,-0.5) {};
\node[circle,fill=black,inner sep=1.5pt,draw] (c) at (6,-0.5) {};
\draw (a) edge[me=3] (b); 
\draw (c) edge[me=7] (b); 
\end{scope}
\end{tikzpicture}   
\caption{The graph $G$.}
\label{metrizedGraph}
\end{figure}
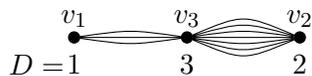

\noindent In \cite{CLM}, it was shown that $D$ has combinatorial rank $2$, but algebraic rank $1$. Here, we strengthen that result, and show that every extension of $D$ to a divisor on a metrized complex is a limit $\mathfrak{g}^2_6$. 

Indeed, let $\mathcal{C}$ be a metrized complex with underlying finite graph $G$ as in Figure \ref{metrizedExample}, where $C_1,C_2,C_3$ are smooth rational curves over $k$.

 
\begin{figure}[H]
\centering

\begin{tikzpicture}[scale=.5]

\coordinate (a) at (0,0);
\coordinate (b) at (9,0);

\draw (a) circle (3cm);
\node at ($(a)+(0,1.5)$) {$C_1$};
\draw[ thick] ($(a) + (-3,0)$) to [out=340,in=200] ($(a)+(3,0)$);
\draw[dashed, thick] ($(a) + (-3,0)$) to [out=20,in=160] ($(a)+(3,0)$);

\draw[fill] ($(a) +(2,1.5)$) circle [radius=0.08];
\draw[fill] ($(b) +(-2,1.5)$) circle [radius=0.08];
\draw[ thick] ($(a) +(2,1.5)$) to [out=20,in=160] ($(b) +(-2,1.5)$);
\draw[fill] ($(a) +(2,1)$) circle [radius=0.08];
\draw[fill] ($(b) +(-2,1)$) circle [radius=0.08];
\draw[ thick] ($(a) +(2,1)$) to [out=20,in=160] ($(b) +(-2,1)$);
\draw[fill] ($(a) +(2,-1)$) circle [radius=0.08];
\draw[fill] ($(b) +(-2,-1)$) circle [radius=0.08];
\draw[ thick] ($(a) +(2,-1)$) to [out=-20,in=200] ($(b) +(-2,-1)$);

\coordinate (a) at (9,0);
\coordinate (b) at (18,0);

\draw (a) circle (3cm);
\node at ($(a)+(0,1.5)$) {$C_3$};
\draw[ thick] ($(a) + (-3,0)$) to [out=340,in=200] ($(a)+(3,0)$);
\draw[dashed, thick] ($(a) + (-3,0)$) to [out=20,in=160] ($(a)+(3,0)$);

\draw[ thick] ($(a) +(2,2)$) to [out=20,in=160] ($(b) +(-2,2)$);
\draw[fill] ($(a) +(2,2)$) circle [radius=0.08];
\draw[fill] ($(b) +(-2,2)$) circle [radius=0.08];

\draw[ thick] ($(a) +(2,1.5)$) to [out=20,in=160] ($(b) +(-2,1.5)$);
\draw[fill] ($(a) +(2,1.5)$) circle [radius=0.08];
\draw[fill] ($(b) +(-2,1.5)$) circle [radius=0.08];

\draw[ thick] ($(a) +(2,1.0)$) to [out=20,in=160] ($(b) +(-2,1)$);
\draw[fill] ($(a) +(2,1.0)$) circle [radius=0.08];
\draw[fill] ($(b) +(-2,1)$) circle [radius=0.08];

\draw[ thick] ($(a) +(2,.5)$) to [out=20,in=160] ($(b) +(-2,.5)$);
\draw[fill] ($(a) +(2,.5)$) circle [radius=0.08];
\draw[fill] ($(b) +(-2,.5)$) circle [radius=0.08];

\draw[ thick] ($(a) +(2,-.5)$) to [out=-20,in=200] ($(b) +(-2,-.5)$);
\draw[fill]  ($(a) +(2,-.5)$) circle [radius=0.08];
\draw[fill] ($(b) +(-2,-.5)$) circle [radius=0.08];

\draw[ thick] ($(a) +(2,-1)$) to [out=-20,in=200] ($(b) +(-2,-1)$);
\draw[fill] ($(a) +(2,-1)$) circle [radius=0.08];
\draw[fill] ($(b) +(-2,-1)$) circle [radius=0.08];

\draw[ thick] ($(a) +(2,-1.5)$) to [out=-20,in=200] ($(b) +(-2,-1.5)$);
\draw[fill]  ($(a) +(2,-1.5)$) circle [radius=0.08];
\draw[fill] ($(b) +(-2,-1.5)$) circle [radius=0.08];

\coordinate (a) at (18,0);

\draw (a) circle (3cm);
\node at ($(a)+(0,1.5)$) {$C_2$};
\draw[ thick] ($(a) + (-3,0)$) to [out=340,in=200] ($(a)+(3,0)$);
\draw[dashed, thick] ($(a) + (-3,0)$) to [out=20,in=160] ($(a)+(3,0)$);

\end{tikzpicture}   
\caption{The metrized complex $\mathcal{C}$.}
\label{metrizedExample}

\end{figure}
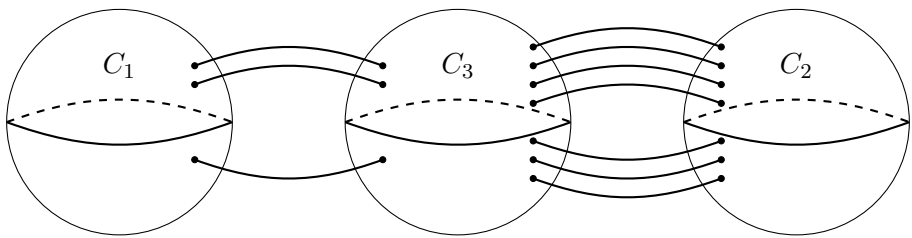

Let $\mathcal{D}$ be a divisor extending $D$. Then $\mathcal{D}$ contains a single chip on $C_1$, three on $C_3$, and two on $C_2$. To show that $\mathcal{D}$ is a limit $\mathfrak{g}^2_6$, we first need to choose  a subspace $H_i\subseteq k(C_i)$ of dimension $3$ for every component $C_i$ of $\mathcal{D}$. Let $x^1,y^1, z^1$ be the points of $C_1$ corresponding to the edges coming from $C_3$, and $x^3,y^3,z^3$ the corresponding points on $C_3$. Note that the property of being a limit $\mathfrak{g}^2_6$ is preserved under linear equivalence of divisors. Therefore, we may assume that $\mathcal{D}|_{C_1} = x^1+y^1$, and $\mathcal{D}|_{C_3} = x^3+y^3+z^3$. Now, choose 

\begin{description}[align=left,style=nextline,leftmargin=55mm]
\item $H_1 = H^0(C_1,\mathcal{O}_{C_1}(x^1+y^1))$,
\item $H_2=H^0(C_2,\mathcal{O}_{C_2}(\mathcal{D}|_{C_2}))$,
\item $H_3 = H^0(C_3,\mathcal{O}_{C_3}(x^3+y^3))$.
\end{description}

Let $\mathcal{E}$ be an effective divisor of degree $2$. By \cite[Corollary A.5]{AB}, we may assume that $\mathcal{E}$ is supported away from the edges. If the degree of $\mathcal{E}|_{C_1}$ is at most $1$,  then one easily finds functions $f_1,f_2,f_3$ in $H_1,H_2,H_3$ such that $\mathcal{D}_{C_i}+ \text{div}(f_i)$ is effective and contains $E_i$ for all $i$. Let $\mathfrak{f}$ be the rational function on $\mathcal{C}$ whose $G$-part is identically zero, and  $C_i$-part  $f_i$ for each $i$. Then $\mathcal{D} + \text{div}(\mathfrak{f})$ is effective and contains $\mathcal{E}$.

Otherwise, $\mathcal{E}$ consists of two chips $p,q$ on $C_1$. Let $f$ be the rational function on $G$ satisfying $\text{div}(f) = 3(v_1)-3(v_3)$. Choose $f_1$ to be the rational function on $C_1$ such that $\text{div}(f_1) = p+q-x^1-y^1$, and choose $f_2$ and $f_3$ as constant functions on $C_2,C_3$ respectively. Then $\mathcal{D}+\text{div}(\mathfrak{f})$ is effective and contains $\mathcal{E}$.


\end{example}

\noindent {\bf Acknowledgements.}
I would like to thank Dustin Cartwright for discussing his upcoming preprint with me, and for many helpful conversations. I am grateful   to  Sam Payne, Matt Baker and to Dhruv Ranganathan for many useful suggestions, and to the referees for their helpful remarks. 
 \bibliographystyle{plain}
\bibliography{bibalgnote}

\end{document}